\documentclass[12pt]{amsart}
\usepackage{amscd}
\usepackage{amsmath}
\usepackage{stmaryrd}
\usepackage{amssymb}
\usepackage{units}
\usepackage[all]{xy}
\def\frk{\frak}               

\def\mm{{\frk m}}

\def\Phi{{\frk n}}
\def\Phi{{\frk N}}
\def\opn#1#2{\def#1{\operatorname{#2}}} 
\opn\chara{char} \opn\length{\ell} \opn\pd{pd} \opn\rk{rk}
\opn\projdim{proj\,dim} \opn\injdim{inj\,dim} \opn\rank{rank}
\opn\depth{depth} \opn\sdepth{sdepth} \opn\fdepth{fdepth}
\opn\grade{grade} \opn\height{height} \opn\embdim{emb\,dim}
\opn\codim{codim}  \opn\min{min} \opn\max{max}

\opn\Tr{Tr} \opn\bigrank{big\,rank}
\opn\superheight{superheight}\opn\lcm{lcm}
\opn\trdeg{tr\,deg}
\opn\reg{reg} \opn\lreg{lreg} \opn\ini{in} \opn\lpd{lpd}
\opn\size{size}
\opn\div{div} \opn\Div{Div} \opn\cl{cl} \opn\Cl{Cl}
\opn\Spec{Spec} \opn\Supp{Supp} \opn\supp{supp} \opn\Sing{Sing}
\opn\Ass{Ass} \opn\Min{Min}
\opn\Ann{Ann} \opn\Rad{Rad} \opn\Soc{Soc}
\opn\Im{Im} \opn\Ker{Ker} \opn\Coker{Coker} \opn\Am{Am}
\opn\Hom{Hom} \opn\Tor{Tor} \opn\Ext{Ext} \opn\End{End}
\opn\Aut{Aut} \opn\id{id}  \opn\deg{deg}

\opn\nat{nat}
\opn\pff{pf}
\opn\Pf{Pf} \opn\GL{GL} \opn\SL{SL} \opn\mod{mod} \opn\ord{ord}
\opn\Gin{Gin} \opn\Hilb{Hilb}
\opn\aff{aff} \opn\con{conv} \opn\relint{relint} \opn\st{st}
\opn\lk{lk} \opn\cn{cn} \opn\core{core} \opn\vol{vol}
\opn\link{link} \opn\star{star}
\opn\gr{gr}

\def\pot#1#2{#1[\kern-0.28ex[#2]\kern-0.28ex]}

\newcommand{\fracs}[2]{\displaystyle\frac{#1}{#2}}

\newcommand{\Sum}[2]{\displaystyle\sum_{#1}^{#2}}

\opn\dirlim{\underrightarrow{\lim}}
\opn\inivlim{\underleftarrow{\lim}}
\let\to=\rightarrow

\def\Implies{\ifmmode\Longrightarrow \else
        \unskip${}\Longrightarrow{}$\ignorespaces\fi}
\def\implies{\ifmmode\Rightarrow \else
        \unskip${}\Rightarrow{}$\ignorespaces\fi}
\def\iff{\ifmmode\Longleftrightarrow \else
        \unskip${}\Longleftrightarrow{}$\ignorespaces\fi}

\let\:=\colon
\newtheorem{Theorem}{Theorem}[]
\newtheorem{Lemma}[Theorem]{Lemma}
\newtheorem{Corollary}[Theorem]{Corollary}
\newtheorem{Proposition}[Theorem]{Proposition}
\theoremstyle{definition}
\newtheorem{Remark}[Theorem]{Remark}

\newtheorem{Definition}[Theorem]{Definition}

\newtheorem{Conjecture}[Theorem]{Conjecture}
\newtheorem{Question}[Theorem]{Question}

\let\epsilon\varepsilon
\let\phi=\varphi
\let\kappa=\varkappa
\def\qed{\ifhmode\textqed\fi
      \ifmmode\ifinner\quad\qedsymbol\else\dispqed\fi\fi}
\def\textqed{\unskip\nobreak\penalty50
       \hskip2em\hbox{}\nobreak\hfil\qedsymbol
       \parfillskip=0pt \finalhyphendemerits=0}
\def\dispqed{\rlap{\qquad\qedsymbol}}

\opn\dis{dis}
\def\pnt{{\raise0.5mm\hbox{\large\bf.}}}

\opn\Lex{Lex}

\begin{document}

\title{\bf Artin approximation property and the General Neron Desingularization}
\author{  Dorin Popescu }
\thanks{ We gratefully acknowledge the support from the project  ID-PCE-2011-1023, granted by the Romanian National Authority for Scientific Research, CNCS - UEFISCDI}

\address{Dorin Popescu, Simion Stoilow Institute of Mathematics of the Romanian Academy, Research unit 5,
University of Bucharest, P.O.Box 1-764, Bucharest 014700, Romania}
\email{dorin.popescu@imar.ro}

\maketitle

\begin{abstract} This is an exposition on the General Neron Desingularization and its applications. We end with a recent constructive form of this desingularization  in  dimension one.
\end{abstract}

\section*{Introduction}

Let $K$ be a field and $R=K\langle x\rangle$, $x=(x_1,\ldots, x_m)$ be the ring of algebraic power series in $x$ over $K$, that is the algebraic closure  of the polynomial ring $K[x]$ in the formal power series ring ${\hat R}=K[[x]]$. Let $f=(f_1,\ldots,f_q) $ in $Y=(Y_1,\ldots,Y_n)$ over $R$ and $\hat y$ be a solution of $f$ in the completion ${\hat R}$ of $R$.
\begin{Theorem}[M. Artin \cite{A1}]\label{ar}  For any $c\in {\bf N}$ there exists a solution $y^{(c)}$ in $R$ such that $y^{(c)}\equiv {\hat y}$  mod $(x)^c$.
\end{Theorem}

In general we say that a local ring $(A,\mm)$ has the {\em  Artin approximation property} if for every system of polynomials $f=(f_1,\ldots,f_q)\in A[Y]^q $, $Y=(Y_1,\ldots,Y_n)$,  a solution $\hat y$ of $f$ in the completion ${\hat A}$ and $c\in {\bf N}$ there exists a solution $y^{(c)}$ in $A$ of $f$ such that $y^{(c)}\equiv {\hat y}$  mod $\mm^c$. In fact $A$
  has the Artin approximation property if every finite system of polynomial equations over $A$ has a solution in $A$ if and only if it has a solution in the completion $\hat A$ of $A$.
 We should mention that M. Artin proved already  in \cite{A} that the ring of  convergent power series with coefficients in $\bf C$ has the Artin approximation property as it was later called.

A ring morphism $u:A\to A'$ of Noetherian rings has  {\em regular fibers} if for all prime ideals $P\in \Spec A$ the ring $\nicefrac{A'}{PA'}$ is a regular  ring, i.e. its localizations are regular local rings.
It has {\em geometrically regular fibers}  if for all prime ideals $P\in \Spec A$ and all finite field extensions $K$ of the fraction field of $\nicefrac{A}{P}$ the ring  $K\otimes_{A/P} \nicefrac{A'}{PA'}$ is regular.

A flat morphism of Noetherian rings $u$ is {\em regular} if its fibers are geometrically regular. If $u$ is regular of finite type then $u$ is called {\em smooth}. A localization of a smooth algebra is called {\em essentially smooth}.

 A Henselian Noetherian local ring $A$ is {\em excellent} if the completion map $A\to \hat A$ is regular. For example, a Henselian discrete valuation ring $V$
is excellent if the completion map $V\to \hat V$ induces a separable fraction field extension.

\begin{Theorem} [M. Artin \cite{A1}]\label{a1} Let $V$ be an excellent Henselian discrete valuation ring and $V\langle x\rangle$ the ring of algebraic power series  in $x$ over $V$,  that is the algebraic closure  of the polynomial ring $V[x]$ in the formal power series ring $V[[x]]$. Then $V\langle x\rangle$ has the Artin approximation property.
\end{Theorem}

 The proof used the so called the N\'eron Desingularization, which says  that  an unramified extension $V\subset V'$ of valuation rings inducing separable field extensions on the fraction and residue fields, is a filtered inductive union of essentially finite type subextensions $V\subset A$, which are regular local rings, even essentially smooth $V$-subalgebras of $V'$.

N\'eron Desingularization is extended by the following theorem.

\begin{Theorem}[General Neron Desingularization, Popescu \cite{P'}, \cite{P}, \cite{P1}, Andr\'e \cite{An}, Teissier \cite{T},  Swan \cite{S}, Spivakovski \cite{Sp}]\label{gnd}  Let $u:A\to A'$ be a  regular morphism of Noetherian rings and $B$ an $A$-algebra of finite type. Then  any $A$-morphism $v:B\to A'$   factors through a smooth $A$-algebra $C$, that is $v$ is a composite $A$-morphism $B\to C\to A'$.
\end{Theorem}

The smooth $A$-algebra $C$ given  for $B$ by the above theorem is called  a {\em General Neron Desingularization}. Note that $C$ is not uniquely associated to $B$ and so we better speak about a General Neron Desingularization.

The above theorem gives a positive answer to a conjecture of  M. Artin  \cite{A2}.

\begin{Theorem}[\cite{P}, \cite{P2}] \label{main} An excellent Henselian local ring has the   Artin approximation property.
\end{Theorem}

This paper is a survey on the Artin approximation property, the General Neron Desingularization and their applications.  It relies mainly on some lectures  given by us within the special semester on Artin Approximation of the Chaire Jean Morlet at CIRM, Luminy, Spring 2015 (see http://hlombardi.free.fr/Popescu-Luminy2015.pdf).

\vskip 0.5 cm

\section{Artin approximation properties}
\vskip 0.5 cm

First we show how one recovers Theorem \ref{main} from Theorem \ref{gnd}. Indeed, let  $f$ be a finite system
of polynomial equations over $A$ in $Y=(Y_1,\ldots,Y_n)$ and $\hat y$ a solution of $f$ in $\hat A$. Set $B=A[Y]/(f)$ and let $v:B\to \hat A$ be the morphism given by $ Y \to \hat y$.
By Theorem \ref{gnd}, $v$  factors through a smooth $A$-algebra $C$, that is $v$ is a composite $A$-morphism $B\to C\to A'$. Thus changing $B$ by $C$ we may reduce the problem to the case when $B$ is smooth over $A$. Since $A'$ is local, changing $B$  by $B_b$ for some $b\in B\setminus v^{-1}(\mm A')$  we may assume that $1\in \big((g):I\big)MB$ for some polynomials $g=(g_1,\ldots,g_r)$ from $(f)$  and  a $r\times r$-minor $M$ of the Jacobian matrix $\left(\fracs{\partial g}{\partial Y}\right)$. Thus $g(\hat y)=0$ and $M(\hat y)$ is invertible. By the Implicit Function Theorem there exists $y\in A$ such that $y\equiv \hat y\ \mbox{modulo}\ \mm\hat A$.

The  following consequence of Theorem \ref{gnd} was noticed and hinted by N. Radu  to M. Andr\'e. This was the origin of Andr\'e's interest to read our theorem and to write later \cite{An}.

\begin{Corollary}   Let $u:A\to A'$ be a  regular morphism of Noetherian rings. Then the differential module $\Omega_{A'/A}$ is flat.
\end{Corollary}
For the proof note that by Theorem \ref{gnd} it follows
that $A'$ is a filtered inductive limit of     some smooth $A$-algebras $C$ and so $\Omega_{A'/A}$ is a filtered inductive limit of $A'\otimes_C\Omega_{C/A}$, the last modules being   free modules.

\begin{Definition}
A Noetherian local ring $(A,\mm)$ has the {\em strong Artin approximation property} if for every finite system of polynomial equations $f$ in $Y=(Y_1,\ldots,Y_n)$ over $A$ there exists a map $\nu:\mathbb N\to \mathbb N$ with the following property:
\begin{quote}
If $y'\in A^n$ satisfies $f(y')\equiv 0\ \mbox{modulo}\ \mm^{\nu(c)}$, $c\in \bf N$, then there exists a solution $y\in A^n$ of $f$ with $y\equiv y'\ \mbox{modulo}\ \mm^c$.
\end{quote}

\end{Definition}
M. Greenberg \cite{Gr}  proved that excellent Henselian discrete valuation rings have the strong Artin approximation property and $\nu$ is linear in this case.
 \begin{Theorem}[M. Artin \cite{A1}] The algebraic power series ring over a field has   the strong Artin approximation property.
 \end{Theorem}
 Note that in general $\nu$ is not linear as it is showed in \cite{Rond}.
The following theorem was conjectured by M. Artin in \cite{A2}.

\begin{Theorem} [\cite{P'''}] \label{p'''} Let $A$ be an excellent Henselian discrete valuation ring and $A\langle x\rangle$ the ring of algebraic power series  in $x$ over $A$. Then $A\langle x\rangle$ has the strong Artin approximation property.
\end{Theorem}

 \begin{Theorem}[Pfister-Popescu \cite{PP} (see also \cite{K}, \cite{P''})] The Noetherian complete local rings have the strong Artin approximation property. In particular, $A$ has the strong Artin approximation property if it has the  Artin approximation property.
  \end{Theorem}
  Thus Theorem \ref{p'''} follows from Theorem \ref{a1} and Theorem \ref{main} gives that excellent Henselian local rings have the strong Artin approximation property.  An easy direct  proof of this fact is given in  \cite[4.5]{P} using Theorem \ref{gnd} and the ultrapower methods.

What about the converse implication in  Theorem \ref{main}? It is clear that $A$ is Henselian if it has the  Artin approximation property.
On the other hand, if $A$ is reduced and it has the  Artin approximation property, then $\hat A$ is reduced, too. Indeed, if $\hat z\in \hat A$ is nonzero and satisfies $\hat z^r=0$ then choosing $c\in \bf N$ such that $\hat z\not\in \mm^c\hat A$  we get a $z\in A$ such that $z^r=0$ and $z\equiv \hat z\ \mbox{modulo}\ \mm^c\hat A$. It follows that $z\not =0$, which contradicts our hypothesis.
It is easy to see  that a local ring $B$ which is finite as a module over  $A$ has the Artin approximation property if $A$ has it. It follows that if $A$ has
the Artin approximation property, then it has  so called reduced formal fibers. In particular, $A$ must be a so called universally japanese ring.

Using also the strong Artin  approximation property it is possible to prove that given a system of polynomial equations $f\in A[Y]^r$, $Y=(Y_1,\ldots,Y_n)$ and another one $g\in A[Y,Z]^t$, $Z=(Z_1,\ldots,Z_s)$ then the sentence
$$L_A:  =\ \  \mbox{there exists}\ y\in A^n \ \ \mbox{such that}\ f(y)=0 \ \mbox{and}\ g(y,z)\not = 0 \ \mbox{for all} \ z\in A^s$$
holds in $A$ if and only if $L_{\hat A}$ holds in $\hat A$ provided that $A$ has the Artin approximation property. In this way it was proved in \cite{BNP} that if $A$ has the Artin approximation property, then $A$ is a normal domain if and only if $\hat A$ is a normal domain, too (this was actually  the starting point of the quoted paper). Later, Cipu and myself \cite{CP} used this fact to show that the formal fibers of $A$ are the so called   geometrically normal domains if $A$ has the Artin approximation property. Finally, Rotthaus \cite{Rot} proved that  $A$ is excellent if $A$ has the Artin approximation property.

Next, let $(A,\mm)$ be an excellent Henselian local ring,  $\hat A$ its completion and MCM$(A)$ (resp. MCM$(\hat A)$) be  the set of isomorphism classes of maximal Cohen Macaulay modules over $A$ (resp. $\hat A$). Assume that $A$ is an isolated singularity.  Then a maximal Cohen-Macaulay module is free on the punctured spectrum. Since $\hat A$ is also an isolated singularity we see that the map $\varphi:$MCM$(A)\to $MCM$(\hat A)$ given by $M\to \hat A\otimes_AM$ is surjective by a theorem of Elkik \cite[Theorem 3]{El}.

\begin{Theorem} [Popescu-Roczen, \cite{PR}] $\varphi$ is bijective.
\end{Theorem}
\begin{proof}
Let $M,N$ be two finite $A$-modules. We may suppose that $M=A^n/(u)$, $N=A^n/(v)$, $u_k=\sum_{j\in [n]} u_{kj} e_j$, $k\in [t]$, $v_r=\sum_{j\in [n]} v_{rj} e_j$, $r\in [p]$, where $u_{kj}, v_{rj}\in A$ and $(e_j)$ is the canonical basis of $A^n$. Let $f:A^n\to A^n$ be an $A$-linear  map defined by an invertible $n\times n$-matrix $(x_{ij})$ with respect to $(e_j)$. Then $f$ induces a  bijection $M\to N$ if and only if $f$ maps $(u)$ onto $(v)$, that is there exist $y_{kr},z_{rk}\in A$, $k\in [t]$, $r\in [p]$ such that

1) $f(u_k)=\sum_{r\in [p]}y_{kr}v_r$, $k\in [t]$ and

2) $f(\sum_{k\in [t]} z_{rk}u_k)=v_r$, $r\in [p].$

Note that 1), 2) are equivalent to

$1')$  $\sum_{i\in [n]}u_{ki}x_{ij}=\sum_{r\in [p]}y_{kr}v_{rj}$,  $k\in [t]$, $j\in [n]$,

$2')$  $\sum_{k\in [t]} z_{rk}(\sum_{i\in [n]}u_{ki}x_{ij})=v_{rj}$,  $r\in [p]$, $j\in [n]$.

Therefore, if  $\hat A\otimes_AM\cong \hat A\otimes _AN$   there exist $(\hat x_{ij})$, $(\hat y_{kr})$, $(\hat z_{rk})$ in $\hat A$ such that $\det(\hat x_{ij})\not \in \mm$ and

$1'')$  $\sum_{i\in [n]}u_{ki}\hat x_{ij}=\sum_{r\in [p]}\hat y_{kr}v_{rj}$,  $k\in [t]$, $j\in [n]$,

$2'')$  $\sum_{k\in [t]} \hat z_{rk}(\sum_{i\in [n]}u_{ki}\hat x_{ij})=v_{rj}$,  $r\in [p]$, $j\in [n]$.

Then by the Artin approximation property there exists a solution of $1'')$, $2'')$, let us say $(x_{ij})$, $( y_{kr})$, $( z_{rk})$ in $ A$, such that $x_{ij}\equiv \hat x_{ij}$, $y_{kr}\equiv \hat y_{kr}$,
$z_{rk}\equiv \hat z_{rk}$ modulo $\mm\hat A$. It follows that $\det(x_{ij})\equiv \det (\hat x_{ij})\not \equiv 0$ modulo $\mm\hat A$ and so $M\cong N$.
\hfill\ \end{proof}

\begin{Corollary} In the hypothesis of the above theorem if $M\in$ MCM$(A)$ is indecomposable, then $\hat A\otimes_AM$ is indecomposable, too.
\end{Corollary}
\begin{proof} Assume that $\hat A\otimes_AM=\hat N_1\oplus \hat N_2$. Then $\hat N_i\in $ MCM$(\hat A)$ and by the surjectivity of $\varphi$ we get $\hat N_i=\hat A\otimes_AN_i$ for some $N_i\in $ MCM$(A)$. Then $\hat A\otimes_AM\cong (\hat A\otimes_AN_1)\oplus (\hat A\otimes_AN_2)$ and the injectivity of $\varphi$
gives $M\cong N_1\oplus N_2$.
\end{proof}

\begin{Remark} { If $A$ is not Henselian then the above corollary is false. For example let $A={\bf C}[X,Y]_{(X,Y)}/(Y^2-X^2-X^3)$. Then $M=(X,Y)A$ is indecomposable in MCM$(A)$ but $\hat A\otimes_AM$ is decomposable. Indeed, for $\hat u=\sqrt{1+X}\in \hat A$ we have $\hat A\otimes_AM=(Y-\hat uX)\hat A\oplus  (Y-\hat uX)\hat A$.}
\end{Remark}
\begin{Remark}{ Let $\Gamma(A)$,  $\Gamma(\hat A)$
 be the so called  AR-quivers of $A$, $\hat A$. Then $\varphi$ induces also an inclusion  $\Gamma(A)\subset\Gamma(\hat A)$ (see \cite{PR}).
}
\end{Remark}

\begin{Remark} {It is  known that MCM$(\hat A)$ is finite if and only if $\hat A$ is a simple singularity. What about a
complex unimodal singularity $R$? Certainly in this case  MCM$(R)$ is infinite but maybe there exists  a special property which characterizes the  unimodal singularities. For this purpose it would be necessary to describe somehow MCM$(R)$ at least in some special cases. Small attempts are done by Andreas Steenpass
\cite{St}.
}
\end{Remark}

For most of the cases when we need the Artin approximation property, it  is  enough to apply Artin's Theorem  \ref{ar}.  Sometimes we might need a special kind of Artin approximation, the so called {\em Artin approximation in nested subring condition}, namely the following result which was also considered as possible by M. Artin in \cite{A2}.

\begin{Theorem}[\cite{P}, {\cite[Theorem 3.6]{P2}}] \label{nes}
 Let $K$ be a field,  $ A=K\langle x\rangle$, $x=(x_1,\ldots,x_m)$,  $f=(f_1,\ldots,f_r)\in K\langle x,Y\rangle^r$, $Y=(Y_1,\ldots,Y_n)$  and  $0\leq s_1\leq \ldots \leq s_n\leq m$, $c$ be some non-negative integers. Suppose that $f$ has a solution ${\hat y}=({\hat y}_1,\ldots,{\hat y}_n)$ in   $ K[[x]] $ such that ${\hat y}_i\in K[[x_1,\ldots,x_{s_i}]]$ for all $1\leq i\leq n$. Then there exists a solution $y=(y_1,\ldots,y_n)$  of $f$ in $A$ such that $y_i\in K\langle x_1,\ldots,x_{s_i}\rangle$ for all $1\leq i\leq n$ and $y\equiv{\hat y} \ \ \mbox{mod}\ \ (x)^cK[[x]]$.
\end{Theorem}
\begin{Corollary} The Weierstrass Preparation Theorem holds for the  ring of  algebraic power series  over a field.
\end{Corollary}
\begin{proof} Let $f\in K\langle x\rangle $, $x=(x_1,\ldots,x_m)$ be an algebraic power series such that $f(0,\ldots,0,x_m)\not =0$. By Weierstrass Preparation Theorem $f$ is associated in divisibility with a monic polynomial ${\hat g}=x_m^p+\sum_{i=0}^{p-1} {\hat z}_ix_m^i\in K[[x_1,\ldots,x_{m-1}]][x_m]$ for some $p\in \bf N$,
${\hat z}_i\in (x_1,\ldots,x_{m-1})K[[x_1,\ldots,x_{m-1}]]$. Thus the system $F_1=f-Y(x_m^p+\sum_{i=0}^{p-1} Z_ix_m^i)$, $F_2=YU-1$ has a solution $\hat y$, $\hat u$, ${\hat z}_i$ in $K[[x]]$ such that ${\hat z}_i\in K[[x_1,\ldots,x_{m-1}]]$. By Theorem \ref{nes} there exists a solution $y,u,z_i$ in $K\langle x\rangle$ such that
$z_i\in K\langle x_1,\ldots,x_{m-1}\rangle$ and is congruent modulo $(x)$ with the previous one. Thus $y$ is invertible and $f=yg$, where $g=x_m^p+\sum_{i=0}^{p-1}  z_ix_m^i\in K\langle x_1,\ldots,x_{m-1}\rangle[x_m]$. By the unicity of the (formal) Weierstrass Preparation Theorem it follows that  $y={\hat y}$ and   $g={\hat g}$.
\hfill\ \end{proof}

Now, we see that Theorem \ref{nes} is useful to get algebraic versal deformations (see \cite{A3}) .  Let $D=K\langle Z\rangle$, $A=K\langle T\rangle/J$, $Z=(Z_1,\ldots,Z_s)$, $T=(T_1,\ldots,T_n)$ and $N=D/(f_1,\ldots,f_d)$. A deformation  of $N$ over $A$ is a \\
$P=K\langle T,Z\rangle/(J)\cong ((A\otimes_KD)_{(T,Z)})^h$-module $L$ such that

1) $L\otimes_AK\cong N$,

2) $L$ is flat over $A$,

\noindent where above $B^h$ denotes the Henselization of a local ring $B$.
The condition 1) says that $L$ has the form $P/(F_1,\ldots,F_d)$ with $F_i\in K\langle T,Z\rangle$, $F_i\cong f_i$ modulo $(T)$ and 2) says that

$2')$ $\Tor_1^A(L,K)=0$

\noindent by the Local Flatness Criterion, since $L$ is $(T)$-adically ideal separated because $P$ is local Noetherian. Let
$$P^e\xrightarrow{\nu} P^d\to P\to L\to 0$$
 be part of a free resolution of $L$ over $P$, where the map $P^d\to P$ is given by $(F_1,\ldots,F_d)$.  Then $2')$ says that tensorizing with $K\otimes_A-$ the above sequence  we get an exact  sequence
 $$D^e\to D^d\to D\to N\to 0,$$
  because $P$ is flat over $A$. Therefore, $2')$ is equivalent to

$2'')$ For all $g\in D^d$ with $\sum_{i=1}^d g_if_i=0$ there exists $G\in K\langle T,Z\rangle^d$ with $G\equiv g$ modulo $(T)$ such that $G\ \mbox{modulo}\ J\in \Im \nu$, that is
$\sum_{i=1}^d G_iF_i\in (J)$.

We would like to construct a versal deformation $L$ (see \cite[pages 157-159]{K}), that is for any $A'=K\langle U\rangle/J'$, $U=(U_1,\ldots,U_{n'})$, $P'=((A'\otimes_KD)_{(U,Z)})^h$ and $L'=K\langle U\rangle/(F')$ a deformation of $N$ to $A'$ there exists a morphism $\alpha:A\to A'$ such that $P'\otimes_PL\cong L'$, where the structural map of $P'$ over $P$ is given by $\alpha$. If we replace above the algebraic power series with formal power series then this problem is solved by Schlessinger in the infinitesimal case followed by some theorems of Elkik and M. Artin.  Set $\hat A=K[[T]]/(J)$, $\hat P=((\hat A\otimes_KD)_{(T,Z)})^h$ . We will assume that we have already $L$ such that  $\hat L= \hat P\otimes_PL $ is versal in the frame of complete local rings. How to get the versal property for $L$ in the frame of algebraic power series?

Let $A',P',L'$ be as above. Since $\hat L$ is versal in the frame of complete local rings, there exists $\hat \alpha:\hat A\to \hat A'$ such that $\hat P'\otimes_{ \hat P}\hat L\cong \hat L'=
\hat P'\otimes_{P'}L'$, where the structure of $\hat P'$ as a $\hat P$-algebra is given by $\hat\alpha$. Assume that $\hat \alpha$ is given by $T\to \hat t\in (U)K[[U]]^n$. Then
we have

i) $J(\hat t)\equiv 0$ modulo $(J')$.

On the other hand, we may suppose that $\hat \alpha$ induces an isomorphism $\hat P'\otimes_{ \hat P}\hat L\to \hat L'$ which is given by $(T,Z)\to (\hat t, \hat z)$ for some $\hat z \in (U,Z)K[[U,Z]]^s$ with $\hat z\equiv Z$ modulo $(U,Z)^2$ and the ideals $(F(\hat t,\hat z))$, $(F')$ of $K[[U,Z]]$ coincide. Thus there exists an invertible  $d\times d$-matrix $\hat C= (\hat C_{ij})$ over $K[[U,Z]]$ with

ii) $F'_i=\sum_{j=1}^d \hat C_{ij} F_j(\hat t,\hat z))$.

By Theorem \ref{nes} we may find $t\in (U)K\langle U\rangle^n$ and  $z\in (U,Z)K\langle U,Z\rangle^s$, $C_{ij}\in K\langle U,Z\rangle$ satisfying i), ii) and such that $t\equiv \hat t$, $z\equiv \hat z$, $C_{ij}\equiv \hat C_{ij}$ modulo $(U,Z)^2$. Note that $\det (C_{ij})\equiv \det \hat C$ modulo $(U,Z)^2$ and so $(C_{ij})$ is invertible.
It follows that $\alpha:A\to A'$ given by $T\to t$ is the wanted one, that is  $P'\otimes_{ P} L\cong  L'$, where the structure of $P'$ as a  $ P$-algebra is given by $\alpha$.

Next we give an idea of the proof of Theorem \ref{nes} in  a particular, but essential case.

\begin{Proposition}  Let $K$ be a field,  $ A=K\langle x\rangle$, $x=(x_1,\ldots,x_m)$,  $f=(f_1,\ldots,f_r)\in K\langle x,Y\rangle^r$, $Y=(Y_1,\ldots,Y_n)$  and  $0\leq s\leq m$, $1\leq q\leq n$,  $c$ be some non-negative integers. Suppose that $f$ has a solution ${\hat y}=({\hat y}_1,\ldots,{\hat y}_n)$ in   $ K[[x]] $ such that ${\hat y}_i\in K[[x_1,\ldots,x_{s}]]$ for all $1\leq i\leq q$. Then there exists a solution $y=(y_1,\ldots,y_n)$  of $f$ in $A$ such that $y_i\in K\langle x_1,\ldots,x_{s}\rangle$ for all $1\leq i\leq q$ and $y\equiv{\hat y} \ \ \mbox{mod}\ \ (x)^cK[[x]]$.
\end{Proposition}
\begin{proof} Note that  $B=K[[x_1,\ldots,x_s]]\langle x_{s+1},\ldots,x_m\rangle$ is excellent Henselian and so it has the Artin approximation property. Thus the system of polynomials \\
$f((\hat y_i)_{1\leq i\leq q}, Y_{q+1},\ldots,Y_n)$ has a solution $(\tilde y_j)_{q<j\leq n}$ in $B$ with $\tilde y_j\equiv \hat y_j$ modulo $(x)^c$. Now it is enough to apply the following lemma for $A=K\langle x_1,\ldots,x_s\rangle$.
\end{proof}
\begin{Lemma}\label{kp} Let $(A,\mm)$ be an excellent Henselian local ring, $\hat A$ its completion, $A[x]^h$, $x=(x_1,\ldots,x_m)$,   ${\hat A}[x]^h$ be the Henselizations of  $A[x]_{(\mm,x)}$ respectively
${\hat A}[x]_{(\mm,x)}$, $f=(f_1,\ldots,f_r)$ a system of polynomials in $Y=(Y_1,\ldots,Y_n)$ over $A[x]^h$ and $1\leq q<n$, $c$ be some positive integers. Suppose that  $f$ has a solution ${\hat y}=({\hat y}_1,\ldots,{\hat y}_n)$ in   ${\hat A}[x]^h$ such that ${\hat y}_i\in {\hat A}$ for all $i\leq q$. Then there exists a solution $y=(y_1,\ldots,y_n)$  of $f$ in $A[x]^h$ such that $y_i\in A$ for all  $i\leq q$ and $y\equiv{\hat y} \ \ \mbox{mod}\ \ \mm^c{\hat A}[x]^h$.
\end{Lemma}

\begin{proof}
 ${\hat A}[x]^h$ is a union of etale neighborhoods of ${\hat A}[x]$.
 Take an etale neighborhood $B$ of ${\hat A}[x]_{(\mm,x)}$ such that ${\hat y}_i\in B$ for all $q<i\leq n$. Then $B\cong ({\hat A}[x,T]/({\hat g}))_{(\mm,x,T)}$
for some monic polynomial $\hat g$ in $T$ over ${\hat A}[x]$ with ${\hat g}(0)\in (\mm,x)$ and $\partial {\hat g}/\partial T(0)\not \in (\mm,x)$, let us say
$${\hat g}=T^e+\sum_{j=0}^{e-1}(\sum_{k\in {\bf N}^m, |k|<u} {\hat z}_{jk}x^k)T^j,$$
 for some $u$ high enough and
${\hat z}_{jk}\in {\hat A}$. Note that ${\hat z}_{00}\in \mm {\hat A}$ and ${\hat z}_{10}\not \in \mm {\hat A}$.
 Changing if necessary $u$, we may suppose that
 $${\hat y}_i\equiv \sum_{j=0}^{e-1}(\sum_{k\in {\bf N}^m, |k|<u}{\hat y}_{ijk}x^k)T^j\ \ \mbox{mod}\ \ {\hat g}$$
  for some ${\hat y}_{ijk}\in {\hat A}$, $q<i\leq n$. Actually, we should take ${\hat y}_i$ as a fraction but for an easier expression we will skip the denominator.
  Substitute $Y_i$, $q<i\leq n$ by
  $$Y_i^+=\sum_{j=0}^{e-1} (\sum_{k\in {\bf N}^m, |k|<u} Y_{ijk}x^k)T^j$$
   in $f$ and divide by the monic polynomial
   $$G=T^e+\sum_{j=0}^{e-1} (\sum_{k\in {\bf N}^m, |k|<u} Z_{jk}x^k)T^j$$
    in ${\hat A}[x,T,Y_1,\ldots,Y_t,(Y_{ij}), (Z_j)]$, where
$(Y_{ijk}), (Z_{jk})$ are new variables.

We get $$f_p(Y_1,\ldots,Y_q,Y^+)\equiv \sum_{j=0}^{e-1} (\sum_{k\in {\bf N}^m, |k|<u} F_{pjk}(Y_1,\ldots,Y_q, (Y_{ijk}), (Z_{jk}))x^kT^j\ \ \mbox{mod}\ \ G,$$
$1\leq p\leq r$. Then $\hat y$ is a solution of $f$ in $B$ if and only if $({\hat y}_1,\ldots,{\hat y}_q,({\hat y}_{ijk}),({\hat z}_{jk}))$ is a solution of $(F_{pjk})$ in ${\hat A}$. As $A$ has the Artin  approximation property we may choose a solution $(y_1,\ldots,y_q, (y_{ijk}),(z_{jk}))$ of $(F_{pjk})$ in $A$
which coincides modulo $\mm^c{\hat A}$ with the former one. Then $$y_i= \sum_{j=0}^{e-1} (\sum_{k\in {\bf N}^m, |k|<u}(y_{ijk}))x^kT^j,$$
 $q<i\leq n$ together with $y_i$, $1\leq i\leq q$ form a solution of $f$ in the etale neighborhood
$B'=(A[x,T]/(g))_{(\mm,x,T)}$,
$$g= T^e+\sum_{i=0}^{e-1} (\sum_{k\in {\bf N}^m, |k|<u} z_{jk}x^k)T^j$$
 of $A[x]_{(\mm,x)}$,
 which is contained in
$A[x]^h$. Clearly, $y$ is the wanted solution.
\end{proof}

\vskip 0.5 cm

\section{Applications to the Bass-Quillen Conjecture}
\vskip 0.5 cm

Let $R[T]$, $T=(T_1,\ldots,T_n)$ be a polynomial algebra in $T$ over a regular local ring $(R,\mm)$. An extension of Serre's Problem proved by Quillen and Suslin is the following

\begin{Conjecture} [Bass-Quillen] \label{bq} Every finitely generated projective module over $R[T]$ is free.
\end{Conjecture}

\begin{Theorem} [Lindel \cite{Li}]\label{l} The Bass-Quillen Conjecture holds if $R$ is essentially of finite type over a field.
\end{Theorem}

Swan's unpublished notes on Lindel's paper  (see \cite[Proposition 2.1]{P3}) contain two interesting remarks.

1) Lindel's proof works also when $R$ is essentially of finite type over a DVR $A$ such that its local parameter $p\not \in \mm^2$.

2) The Bass-Quillen Conjecture holds if $(R,\mm)$ is a regular local ring containing a field, or $p=$char $R/\mm\not\in \mm^2$ providing that the following question has a positive answer.

\begin{Question} [Swan]\label{s}  Is a regular local ring  a filtered inductive limit of regular local rings essentially of finite type over $\bf Z$?
 \end{Question}

 Indeed, suppose for example that $R$ contains a field and $R$ is a filtered inductive limit of regular local rings $R_i$ essentially of finite type over a prime field $P$. A finitely generated projective $R[T]$-module $M$ is an extension of a  finitely generated projective $R_i[T]$-module $M_i$ for some $i$, that is $M\cong R[T]\otimes_{R_i[T]}M_i$. By  Theorem \ref{l} we get $M_i$ free and so $M$ is free, too.

\begin{Theorem} [\cite{P3}] \label{sp} Swan's Question \ref{s}
 holds for regular local rings $(R,\mm,k)$ which are in one of the following cases:
\begin{enumerate}
\item{} $R$ contains a field,

\item{} the characteristic $p$ of $k$  is not in $\mm^2$,

\item{} $R$ is excellent Henselian.
\end{enumerate}
\end{Theorem}

\begin{proof}
(1) Suppose that $R$ contains a field $k$. We may assume that $k$ is the prime field of $R$ and so a perfect field. Then the inclusion $u:k\to R$ is  regular  and
by Theorem \ref{gnd} it is a filtered inductive limit of smooth $k$ morphisms $k\to R_i$. Thus $R_i$ is a regular ring  of finite type over $k$ and so over $\bf Z$. Therefore  $R$ is a filtered inductive limit of regular local rings essentially of finite type over $\bf Z$. Similarly we may treat (2).

(3) First assume that $R$ is complete. By the Cohen Structure Theorem we may also assume that $R$ is a factor of a complete local ring of type $A={\bf Z}_{(p)}[[x_1,\ldots,x_m]]$ for some prime integer $p$. By  (2) we see that $A$ is a filtered inductive limit of regular local rings $A_i$ essentially of finite type over $\bf Z$. Since $R, A$ are regular local rings we see that $R=A/(x)$ for a part $x$ of a regular system of parameters of $A$. Then there exists a system of elements $x'$ of a certain $A_i$ which is mapped into $x$ by the limit map $A_i\to A$. It follows that $x'$ is  part of a regular system of parameters of $A_t$ for all $t\geq j$ for some $j>i$ and so $R_t=A_t/(x')$ are regular local rings. Now, it is enough to see that  $R$ is a filtered inductive limit of $R_t$, $t\geq j$.

Next assume that $R$ is excellent Henselian and let $\hat R$ be its completion. Using \cite{AD}, or \cite{S} it is enough to show that given a finite type $\bf Z$-subalgebra $E$ of $R$ the inclusion $\alpha:E\to R$ factors through a regular local ring $E'$ essentially of finite type over $\bf Z$, that is there exists $\beta:E'\to R$ such that $\alpha$ is the composite map $E\to E'\xrightarrow{\beta} R$.

As above, $\hat R$ is a filtered inductive limit of regular local rings and so the composite map ${\hat \alpha}:E\xrightarrow{\alpha} R\to \hat R$ factors through a regular local ring $F$ essentially of finite type over $\bf Z$. We may choose a finite type $\bf Z$-subalgebra $D\subset F$ such that $F\cong D_q$ for some $q\in \Spec D$ and the map $E\to F$ factors through $D$, i.e. $D$ is an $E$-algebra. As $D$ is excellent, its regular locus Reg $D$ is open and so there exists $d\in D\setminus q$ such that $D_d$ is a regular ring. Changing $D$ by $D_d$ we may assume that $D$ is regular.

Let $E={\bf Z}[b_1,\ldots,b_n]$ for some $b_i\in E\subset R$ and let $D={\bf Z}[Y]/(h)$, $Y=(Y_1,\ldots,Y_n)$ for some polynomials $h$. Since $D$ is an $E$-algebra we may write $b_i\equiv P_i(Y)\  \mbox{modulo}\ h$, $i=1,\ldots,n$ for some polynomials $P_i\in  E[Y]\subset R[Y] $. Note that there exists ${\hat y}\in {\hat R}^n$ such that $b_i=P_i({\hat y})$, $h({\hat y})=0$ because $\hat \alpha$ factors through  $D$. As $R$ has the  Artin approximation property, by Theorem \ref{main} there exists $y\in R^n$ such that $b_i=P_i(y)$, $h(y)=0$. Let $\rho:D\to R$ be the map given by $Y\to y$. Clearly, $\alpha$ factors through $D$ and we may take $E'=D_{\rho^{-1}\mm}$. More precisely, we have the following diagram which is commutative except in the right square,

  $$
  \begin{xy}\xymatrix{E \ar[dr] \ar[r]^{\alpha} & R \ar[r] & \hat R\\
  & D \ar[u]_{\rho} \ar[r]& F \ar[u]}
  \end{xy}
  $$

\hfill\  \end{proof}

\begin{Corollary} [\cite{P3}]\label{p3} The Bass-Quillen Conjecture holds if $R$ is a regular local ring in one of the cases (1), (2) of the above theorem.
\end{Corollary}

\begin{Remark} Theorem \ref{sp} is not a complete answer to  Question \ref{s}, but (3) says that a positive answer is expected in general. Since there exists no  result similar to Lindel's saying that the Bass-Quillen Conjecture holds for all regular local rings essentially of finite type over $\bf Z$ we decided to wait with our further research. So  we have waited  already 25 years.
\end{Remark}

Another problem is to replace in the Bass-Quillen Conjecture the polynomial algebra $R[T]$ by other $R$-algebras. The tool is given by   the following theorem.

\begin{Theorem} [Vorst \cite{V}]\label{v} Let $A$ be a ring, $A[x]$, $x=(x_1,\ldots,x_m)$ a polynomial algebra, $I\subset A[x]$ a monomial ideal and $B=A[x]/I$. Then every finitely generated projective $B$-module $M$ is extended  from a finitely generated projective $A$-module $N$, that is $M\cong B\otimes_AN$, if for all $n\in \bf N$ every  finitely generated projective $A[T]$-module, $T=(T_1,\ldots,T_n)$ is extended from a  finitely generated projective $A$-module.
\end{Theorem}

\begin{Corollary}[\cite{P6}] Let $R$ be a regular local ring in one of the cases (1), (2) of Theorem \ref{sp}, $I\subset R[x]$ be a monomial ideal with $x=(x_1,\ldots,x_m)$ and $B=R[x]/I$. Then any finitely generated projective $B$-module is free.
\end{Corollary}

For the proof, apply the above theorem using Corollary \ref{p3}.

The Bass-Quillen Conjecture could also hold when $R$ is not regular as  the following corollary shows.
\begin{Corollary}[\cite{P6}]\label{m2} Let $R$ be a regular local ring  in one of the cases  of Theorem \ref{sp},   $I\subset R[x]$  be a monomial ideal with $x=(x_1,\ldots,x_m)$
and $B=R[x]/I$. Then every finitely generated projective $B[T]$-module, $T=(T_1,\ldots, T_n)$ is free.
\end{Corollary}
 This result holds because $B[T]$ is a factor of $R[x,T]$ by the monomial ideal $IR[x,T]$.

\begin{Remark} If $I$ is not monomial, then the Bass-Quillen Conjecture may fail when replacing $R$ by $B$. Indeed, if $B={\bf R}[x_1,x_2]/(x_1^2-x_2^3)$ then there exist finitely generated projective $B[T]$-modules of rank one which are not free (see \cite[(5.10)]{La}).
\end{Remark}

Now, let $(R,\mm)$ be a regular local ring and $f\in \mm\setminus \mm^2$.

\begin{Question} [Quillen \cite{Q}] \label{q} Is free a  finitely generated projective  module over $R_f$?
\end{Question}

\begin{Theorem} [Bhatwadeckar-Rao, \cite{BR}]\label{br}  Quillen's Question   has a positive answer if $R$ is essentially of finite type over a field.
\end{Theorem}

\begin{Theorem} [\cite{P5}]\label{p5}  Quillen's Question   has a positive answer if $R$ contains a field.
\end{Theorem}

This goes   similarly to Corollary \ref{p3} using Theorem \ref{br} instead of Theorem \ref{l}.

\begin{Remark} The paper \cite{P5} was not accepted for publication in many journals since  the referees said that "relies on a theorem [that is Theorem \ref{gnd}] which is still not recognized by the mathematical community". Since our paper was quoted as an unpublished preprint in \cite{S}  we published it later in the Romanian Bulletin and  it was noticed and quoted by many people (see for instance \cite{KV}).
\end{Remark}
\vskip 0.5 cm
\section{General Neron Desingularization}
\vskip 0.5 cm
Using Artin's methods from \cite{A}, Ploski gave the following theorem, which is the first form of a possible extension of Neron Desingularization in $\dim >1$.

\begin{Theorem} [\cite{Pl}] \label{pl} Let ${\bf C}\{x\}$, $x=(x_1,\ldots,x_m)$, $f=(f_1,\ldots,f_s)$ be some convergent power series from ${\bf C}\{x,Y\}$, $Y=(Y_1,\ldots,Y_n)$ and $\hat y\in {\bf C}[[x]]^n$ with $\hat y(0)=0$ be a solution of $f=0$. Then the map $v:B={\bf C}\{x,Y\}/(f)\to {\bf C}[[x]]$ given by $Y\to \hat y$ factors through an $A$-algebra of type $B'={\bf C}\{x,Z\}$ for some variables $Z=(Z_1,\ldots,Z_s)$, that is $v$ is a composite map $B\to B'\to {\bf C}[[x]]$.
\end{Theorem}

Using Theorem \ref{gnd} one can get an extension of the above theorem.

\begin{Theorem} [\cite{P6}] Let $(A,\mm)$ be an excellent Henselian local ring, $\hat A$ its completion,  $B$ a finite type $A$-algebra and $v:B\to \hat A$ an $A$-morphism.  Then  $v$  factors through an $A$-algebra of type $ A[Z]^h$ for some variables $Z=(Z_1,\ldots,Z_s)$, where $A[Z]^h$ is the Henselization of $A[Z]_{(\mm,Z)}$.
\end{Theorem}

   Suppose that $B=\nicefrac{A[Y]}{I}$, $Y=(Y_1,\ldots,Y_n)$. If $f=(f_1,\ldots,f_r)$, $r\leq n$ is a system of polynomials from $I$,  then denote by $\Delta_f$ the  ideal generated by  all $r\times r-$minors of the Jacobian matrix $\left(\fracs{\partial f_i}{\partial Y_j}\right)$. After Elkik \cite{El}, let $H_{B/A}$ be the radical of the ideal $\Sum{f}{} \big((f):I\big)\Delta_fB$, where the sum is taken over all systems of polynomials $f$ from $I$ with $r\leq n$.
Then $B_P$, $P\in \Spec B$ is essentially smooth over $A$ if and only if $P\not \supset H_{B/A}$ by the Jacobian criterion for smoothness.
   Thus  $H_{B/A}$ measures the non smooth locus of $B$ over $A$.

In the linear case we may  easily  get cases of Theorem \ref{gnd} when $\dim A>1$.

\begin{Lemma} [{\cite[(4.1)]{P'}}] Let $A$ be a ring and $a_1,a_2$ a weak regular sequence of $A$, that is $a_1$ is a non-zero divisor of $A$ and $a_2$ is a non-zero divisor of $A/(a_1)$. Let $A'$ be a flat $A$-algebra and set $B=A[Y_1,Y_2]/(f)$, where $f=a_1Y_1+a_2Y_2$. Then $H_{B/A}$ is the radical of $(a_1,a_2)$ and any $A$-morphism $B\to A'$ factors through a polynomial $A$-algebra in one variable.
   \end{Lemma}
   \begin{proof} Note that all solutions of $f=0$ in $A$ are multiples of $(-a_2,a_1)$. By  flatness,  any solution of $f$ in $A'$ is a linear combinations of some solutions of $f$ in $A$ and so again a multiple of $(-a_2,a_1)$.
Let $h:B\to A'$ be a map given by $Y_i\to y_i\in A'$. Then $(y_1,y_2)=z(-a_2,a_1)$ and so $h$ factors through $A[Z]$, that is $h$ is the composite map $B\to A[Z]\to A'$, the first map being given by $(Y_1,Y_2)\to Z(-a_2,a_1)$ and the second one by $Z\to z$.
\hfill\    \end{proof}

   \begin{Proposition} [{\cite[Lemma 4.2]{P'}}] Let $f_i=\sum_{i=1}^n a_{ij}Y_j\in A[Y_1,\ldots,Y_n]$, $i\in [r]$ be a system of linear homogeneous polynomials and $y^{(k)}=(y^{(k)}_1,\ldots,y^{(k)}_n)$, $k\in [p]$ be a complete system of solutions of $f=(f_1,\ldots,f_r)=0$ in $A$. Let $b=(b_1,\ldots,b_r)\in A^r$ and $c$ a solution of $f=b$ in $A$. Let $A'$ be a flat $A$-algebra and \\
   $B= A[Y_1,\ldots,Y_n]/(f-b)$. Then  any $A$-morphism $B\to A'$ factors through a polynomial $A$-algebra in $p$ variables.
   \end{Proposition}
\begin{proof} Let $h:B\to A'$ be a map given by $Y\to y'\in A'^n$. Since $A'$ is flat over $A$ we see that  $y'-h(c)$ is a linear combinations of $y^{(k)}$, that is there exists $z=(z_1,\ldots,z_p)\in A'^p$ such that $y'-h(c)=\sum_{k=1}^pz_kh(y^{(k)})$. Therefore, $h$ factors through $A[Z_1,\ldots,Z_p]$,
that is $h$ is the composite $A$-morphism $B\to A[Z_1,\ldots,Z_p]\to A'$, where the first map is given by $Y\to c+\sum_{k=1}^pZ_ky^{(k)}$ and the second one by $Z\to z$.
\hfill\ \end{proof}

Another   form of Theorem \ref{gnd} is the following theorem which is a positive answer  to a conjecture of M. Artin \cite{A4}.

\begin{Theorem} [Cipu-Popescu \cite{CP1}]\label{cp1} Let $u:A\to A'$ be a  regular morphism of Noetherian rings, $B$ an  $A$-algebra of finite type, $v:B\to A'$  an $A$-morphism and $D\subset \Spec B$ the open smooth locus of $B$ over $A$. Then there exist    a smooth $A$-algebra $C$ and two $A$-morphisms $t:B\to C$, $w:C\to A'$ such that $v=wt$ and $C$ is smooth over $B$ at ${t^*}^{-1}(D)$, $t^*:\Spec C\to \Spec B$ being induced by $t$.
\end{Theorem}

There exists also a form of Theorem \ref{gnd} recalling us  the strong Artin approximation property.

\begin{Theorem} [\cite{P4}, \cite{P6}]\label{corm} Let $(A,\mm) $ be a Noetherian local ring with the completion map $A\to \hat A$ regular,  $B$ an $A$-algebra  of finite type and  $\nu$ the Artin function over $\hat A$ associated to the system of polynomials $f$ defining $B$.  Then there exists a function $\lambda:\bf N\to \bf N$, $\lambda\geq \nu$ such that for every positive integer $c$ and every morphism $v:B\to A/\mm^{\lambda(c)}$ there exists a smooth $A$-algebra $C$ and two $A$-algebra morphisms $t:B\to C$, $w:C\to A/\mm^c$ such that   $wt$ is the composite map $B\xrightarrow{v} A/\mm^{\lambda(c)}\to A/\mm^c$.
\end{Theorem}

Sometimes, we may find some information about $\lambda$ (and so about $\nu$).
Let  $A$ be a discrete valuation ring, $x$ a local parameter of $A$,  $A'=\hat A$ its completion and  $B=A[Y]/I$, $Y=(Y_1,\ldots,Y_n)$ an  $A$-algebra of finite type.
  If $f=(f_1,\ldots,f_r)$, $r\leq n$ is a system of polynomials from $I$ then we consider a $r\times r$-minor $M$ of the Jacobian matrix $(\partial f_i/\partial Y_j)$.  Let  $c\in \bf N$. Suppose that there exists  an $A$-morphism $v:B\to A'/(x^{2c+1})$   and  $N\in ((f):I)$ such that
 $v(NM)\not\in (x)^c/(x^{2c+1})$, where for simplicity we write $v(NM)$ instead of $v(NM+I)$.

\begin{Theorem}[{\cite[Theorem 10]{P6}}]\label{lambda} There exists a $B$-algebra $C$ which is smooth over $A$ such that every $A$-morphism $v':B\to A'$ with $v'\equiv v \ \mbox{modulo}\ x^{2c+1}$ (that is $v'(Y)\equiv v(Y) \ \mbox{modulo}\
 x^{2c+1}$) factors through $C$.
\end{Theorem}

\begin{Corollary} [{\cite[Theorem 15]{P6}}]\label{lambda1} In the assumptions and notation of Corollary \ref{lambda} there exists a canonical bijection
$$A'^s\to \{v'\in \Hom_A(B,A'):v'\equiv v \ \mbox{modulo}\ x^{2c+1}\}$$
 for some $s\in \bf N$.
 \end{Corollary}

 Let $k$ be a field and $F$  a  $k$-algebra of finite type, let us say $F=k[U]/J$, $U=(U_1,\ldots,U_n)$. An arc $\Spec k[[x]]\to \Spec F$  is given by a $k$-morphism $F\to A'=k[[x]]$.
Assume that $H_{F/k}\not =0$ (this happens for example when $F$ is reduced and $k$ is perfect). Set $A=k[x]_{(x)}$, $B=A\otimes_kF$. Let
   $f=(f_1,\ldots,f_r)$, $r\leq n$ be a system of polynomials from $J$ and $M$ a $r\times r$-minor  of the Jacobian matrix $(\partial f_i/\partial U_j)$.  Let  $c\in \bf N$. Assume that there exists  an $A$-morphism $g:F\to A'/(x^{2c+1})$   and  $N\in ((f):J)$ such that
 $g(NM)\not\in (x)^c/(x^{2c+1})$.
 Note that $A\otimes_k-$ induces a bijection $\Hom_k(F,A')\to \Hom_A(B,A')$ by adjunction.

 \begin{Corollary} [{\cite[Corollary 16]{P6}}]The set $\{g'\in \Hom_k(F,A'):g'\equiv g \ \mbox{modulo}\ x^{2c+1}\}$ is in bijection with an affine space $A'^s$ over $A'$ for some $s\in \bf N$.
 \end{Corollary}

Next we give a possible extension of Greenberg's result on the strong Artin approximation property \cite{Gr}.
Let $(A,\mm)$ be a Cohen-Macaulay local ring (for example a reduced ring) of dimension one, $A'={\hat A}$ the completion of $A$, $B=A[Y]/I$, $Y=(Y_1,\ldots,Y_n)$ an  $A$-algebra of finite type and $c,e\in \mathbb N$.
Suppose that there exists  $f=(f_1,\ldots,f_r)$ in $I$, a $r\times r$-minor $M$ of the Jacobian matrix $(\partial f/\partial Y)$, $N\in ((f):I) $ and an $A$-morphism $v:B\to A/\mm^{2e+c}$
such that $(v(MN))\supset \mm^e/\mm^{2e+c}$. Then we may construct a General Neron Desingularization in the idea of Theorem \ref{lambda}, which could be used to get the following theorem.

\begin{Theorem}[A. Popescu-D. Popescu, \cite{AP}] \label{gr} There exists an $A$-morphism $v':B\to {\hat A}$ such that $v'\equiv v\ \mbox{modulo}\ \mm^c\hat A$, that is $v'(Y+I)\equiv v(Y+I)\ \mbox{modulo}\ \mm^c\hat A$. Moreover, if $A$ is also excellent Henselian  there exists an $A$-morphism $v':B\to  A$ such that $v'\equiv v\ \mbox{modulo}\ \mm^c$.
\end{Theorem}
\begin{Remark} The above theorem could be extended for Noetherian local rings of dimension one (see \cite{PP1}). In this case the statement depends also on a reduced primary decomposition of $(0)$ in $A$.
\end{Remark}

Using \cite{AP} we end this section with an algorithmic attempt  to explain the proof of Theorem \ref{gnd} in the frame of Noetherian local domains of dimension one. Let $u:A\to A'$ be a flat morphism of Noetherian local domains of dimension $1$. Suppose that $A\supset \mathbb Q$  and  the maximal ideal $\mm$ of $A$ generates the maximal ideal of $A'$. Then $u$ is a regular morphism. Moreover, we suppose that there exist canonical inclusions $k=A/\mm\to A$, $k'=A'/\mm A'\to A'$ such that $u(k)\subset k'$.

 If $A$ is  essentially of finite type over $\mathbb Q$, then the ideal $H_{B/A}$ can be  computed in \textsc{Singular} by following its definition but it is easier to describe only the ideal $\sum_f \big((f):I\big)\Delta_fB$ defined above.  This is the case  considered in our algorithmic part, let us say $A\cong k[x]/F$ for some variables $x=(x_1,\ldots x_m)$, and the completion of $A'$ is $k'\llbracket x\rrbracket/(F)$. When $v$ is defined by polynomials $y$ from $k'[x]$ then our problem is easy. Let $L$ be the field obtained by adjoining to $k$ all coefficients of $y$. Then $R=L[x]/(F)$ is a subring of $A'$ containing $\Im v$ which is essentially smooth over $A$. Then we may take $B' $ as a standard smooth $A$-algebra such that $R$ is a localization of $B'$.
Consequently we  suppose usually  that $y$ is not in $k'[x]$.

We may suppose that $v(H_{B/A})\not =0$. Indeed, if $v(H_{B/A}) =0$ then $v$ induces an $A$-morphism $v':B'=B/H_{B/A}\to A'$ and we may replace $(B,v)$ by $(B',v')$. Applying this trick several times we reduce to the case  $v(H_{B/A})\not =0$. However, the fraction field of $\Im v$ is essentially smooth over $A$ by separability, that is $H_{\Im v/A}A'\not =0$ and in the worst case our trick will change $B$ by $\Im v$ after several steps.

  Choose $P'\in \big(\Delta_f((f):I)\big)\setminus I$ for some system of polynomials $f=(f_1,\ldots,f_r)$ from $I$ and  $d'\in \big(v(P')A'\big )\cap A$, $d'\not = 0$. Moreover, we may choose $P'$ to be from $M\big((f):I\big)$ where $M$ is a $r\times r$-minor of $\left(\fracs{\partial f}{\partial Y}\right)$. Then $d'=v(P')z\in \big(v(H_{B/A})\big)\cap A$ for some $z\in A'$. Set $B_1=B[Z]/(f_{r+1}) $, where $f_{r+1}=-d'+P'Z$
and let $v_1:B_1\to A'$ be the map of $B$-algebras given by $Z\to z$. It follows that $d'\in \big(( f,f_{r+1}):(I,f_{r+1})\big)$ and  $d'\in \Delta_f$, $d'\in \Delta_{f_{r+1}}$. Then $d=d'^2\equiv P\ \mbox{modulo}\ (I,f_{r+1})$ for $P=P'^2Z^2\in H_{B_1/A}$. Replace $B$ by $B_1$ and the Jacobian matrix $J=(\partial f/\partial Y)$ will be  now the new $J$ given by $
\left(\begin{array}{cc}
J & 0 \\
* & P'
\end{array}\right).$
Thus we reduce to the case when $d\in H_{B/A}\cap A$.

But how to get $d$ with a computer if $y$ is not polynomial defined over $k'$? Then the algorithm is complicated because we are not able to tell   the computer who is $y$ and so  how to get $d'$. We may choose an element $a\in \mm$ and find a minimal $c\in \mathbb N$ such that $a^c\in (v(M))+(a^{2c}) $ (this is possible because $\dim A=1$). Set $d'=a^c$. It follows that $d'\in (v(M))+ (d'^2)\subset  (v(M))+ (d'^4)\subset \ldots $ and so $d'\in (v(M))$, that is $d'=v(M)z$ for some $z\in A'$. Certainly we cannot find precisely $z$, but later it is enough  to know just a kind of truncation of it modulo $d'^6$.

Thus we may suppose that there exist  $f=(f_1,\ldots,f_r)$, $r\leq n$ a system of polynomials from $I$, a $r\times r$-minor $M$ of the Jacobian matrix $(\partial f_i/\partial Y_j)$, $N\in ((f):I)$ such that $0\not = d\equiv P=MN\ \mbox{modulo}\ I$.  We may assume that $M=\det((\partial f_i/\partial Y_j)_{i,j\in [r]})$. Set   $\bar A=A/(d^3)$, $\bar A'=A'/d^3A'$, $\bar u=\bar A\otimes_Au$, $\bar B=B/d^3B$, $\bar v=\bar A\otimes_Av$. Clearly, $\bar u$ is a regular morphism of Artinian  local rings and it is easy to find a General Neron Desingularization in this frame. Thus there exists a $\bar B$-algebra $C$, which is smooth over $\bar A$  such that $\bar v$ factors through $C$. Moreover we may suppose that $C\cong  (\bar A[U]/(\omega))_{\tau}$ for some polynomials $\omega\tau\in k[U]$ which are not in $m (\bar A[U]/(\omega))$ (note that $k\subset A$).
Then $D\cong  ( A[U]/(\omega))_{\tau}$ is smooth over $A$ and $u$ factors through $D$. Usually, $v$ does not factor through $D$, though $\bar v$ factors through
$C\cong {\bar A}\otimes_AD$.

 Let  $y'\in D^n$ be such that the composite map $\bar B\to C\to \bar D$ is given by $Y\to y'+d^3D$. Thus $I(y')\equiv 0$ modulo $d^3D$.
We have $d\equiv P$ modulo $I$ and so $P(y')\equiv d$ modulo $d^3$ . Thus $P(y')=ds$ for a certain $s\in D$ with $s\equiv 1$ modulo $d$.
 Let  $H$ be the $n\times n$-matrix obtained by  adding down to $(\partial f/\partial Y)$ as a border the block $(0|\mbox{Id}_{n-r})$. Let $G'$ be the adjoint matrix of $H$ and $G=NG'$. We have
$$GH=HG=NM \mbox{Id}_n=P\mbox{Id}_n$$
and so
$$ds\mbox{Id}_n=P(y')\mbox{Id}_n=G(y')H(y').$$
 Set
 $h=s(Y-y')-dG(y')T$,
 where  $T=(T_1,\ldots,T_n)$ are new variables. Since
$$s(Y-y')\equiv dG(y')T\ \mbox{modulo}\ h$$
and
$$f(Y)-f(y')\equiv \sum_j\fracs{\partial f}{\partial Y_j}(y') (Y_j-y'_j)$$
modulo higher order terms in $Y_j-y'_j$, by Taylor's formula we see that for $p=\max_i \deg f_i$ we have
$$s^pf(Y)-s^pf(y')\equiv  s^{p-1}dP(y')T+d^2Q$$
modulo $h$, where $Q\in T^2 D[T]^r$. This is because $(\partial f/\partial Y)G=(P\mbox{Id}_r|0)$.  We have $f(y')=d^2b$ for some $b\in dD^r$. Set
$g_i=s^pb_i+s^pT_i+Q_i, \  i\in [r]$. Then we may take $B'$ to be a localization of $(D[Y,T]/(I,h,g))_s$.

\begin{Remark} An algorithmic proof in the frame of all Noetherian local rings of dimension one is given in  \cite{PP1}.
\end{Remark}

\end{document}